\documentclass[12pt,leqno]{article}
\usepackage[shortlabels]{enumitem}
\usepackage{preamble}

\title{Positively curved Finsler metrics on vector bundles}
\author{Kuang-Ru Wu}
\newcommand{\RN}[1]{%
  \textup{\uppercase\expandafter{\romannumeral#1}}%
}

\usepackage{fullpage}
\usepackage{tikz-cd}

\usepackage{fancyhdr}

\theoremstyle{plain}

\numberwithin{equation}{section}

\begin{document}

\date{}

\parskip=6pt

\maketitle
\begin{abstract}
We construct a convex and strongly pseudoconvex Kobayashi positive Finsler metric on a vector bundle $E$ under the assumption that the symmetric power of the dual $S^kE^*$ has a Griffiths negative $L^2$-metric for some $k$.
The proof relies on the negativity of direct image bundles and the Minkowski inequality for norms. As a corollary, we show that given a strongly pseudoconvex Kobayashi positive Finsler metric, one can upgrade to a \textit{convex} Finsler metric with the same property.  We also give an extremal characterization of Kobayashi curvature for Finsler metrics.
\end{abstract}

\section{Introduction}
Let $E$ be a holomorphic vector bundle of rank $r$ over a compact complex manifold $X$ of dimension $n$. Denote by $E^*$ the dual bundle of $E$. Let $P(E^*)$ be the projectivized bundle of $E^*$, $O_{P(E^*)}(-1)$ the tautological line bundle over $P(E^*)$, and denote by $O_{P(E^*)}(1)$ the dual bundle of $O_{P(E^*)}(-1)$. A vector bundle $E$ is said to be ample if the line bundle $O_{P(E^*)}(1)$ is ample, see \cite{Hart66}.

Meanwhile, for a Hermitian holomorphic vector bundle $(E,H)$, Griffiths \cite{Griff69} introduced a notion of positivity using the curvature of the metric $H$, which is now called Griffiths positivity. When $E$ is a line bundle, the existence of a Griffiths positive Hermitian metric on $E$ is equivalent to the ampleness of $E$; essentially by the Kodaira embedding theorem. When $E$ is of higher rank, the existence of a Griffiths positive metric implies ampleness (see e.g. \cite[p. 89]{MR909698} or \cite[p. 283]{demailly1997complex}); however, the converse statement known as the Griffiths conjecture is still open (see earlier results \cite{Umemura,CampanaFlenner} and recent developments \cite{naumann2017approach,demailly2020hermitianyangmills,pingali2021note,finski2020monge}). 

On the other hand, Kobayashi \cite{Negfinsler} defined a notion of curvature for strongly pseudoconvex Finsler metrics on holomorphic vector bundles, and he showed that the dual bundle $E^*$ is ample if and only if $E$ has a strongly pseudoconvex Kobayashi negative Finsler metric. Moreover, he raised the question whether the ampleness of $E$ is equivalent to the existence of a strongly pseudoconvex Kobayashi positive Finsler metric on $E$. One direction is proved in \cite{feng2018complex} where the authors show the existence of such a Finsler metric on $E$ implies the ampleness of $E$.

The motivation of this paper is twofold. One is the remaining part of the Kobayashi conjecture, and the other is whether Kobayashi's conjecture implies Griffiths' conjecture. Although we do not solve either of the problems, some progress is made in the results below. Before we state our results, we recall a natural way to produce Hermitian metrics on direct image bundles. 

Let $h$ be a Hermitian metric on $O_{P(E)}(1)$ and $\{\omega_z\}_{z\in X}$ be a family of Hermitian metrics on the fibers $P(E_z)$. Given a positive integer $k$, there is a canonical isomorphism $$\Phi_{k,z}:S^kE^*_z\to H^0(P(E_z),O_{P(E_z)}(k)), $$ see for example \cite[p. 278, Theorem 15.5]{demailly1997complex}. We can define an $L^2$-metric $H_k$ on $S^kE^*$ using $h^k$ and $\{\omega_z\}$
\begin{align}\label{L2 metric}
H_k(u,v):=\int_{P(E_z)} h^k(\Phi_{k,z}(u),\Phi_{k,z}(v) ) \omega_z^{r-1}, \textup{ for $u \textup{ and } v\in S^kE^*_z$.}
\end{align}
\begin{theorem}\label{thm dual}
If the Hermitian metric $H_k$ in (\ref{L2 metric}) is Griffiths negative for some $k$, then $E^*$ admits a convex Finsler metric that is strongly plurisubharmonic on  $E^*\setminus \{\text{zero section}\}$; moreover, the bundle $E$ admits a convex and strongly pseudoconvex Finsler metric with positive Kobayashi curvature.
\end{theorem}

The idea behind the proof of Theorem \ref{thm dual} is straightforward. The $k$-th root
of $H_k$ induces a Finsler metric $F_k$ on $E^*$. This Finsler metric is convex by the Minkowski inequality, and it is strongly plurisubharmonic on $E^*\setminus \{\text{zero section}\}$ by Griffiths negativity. The convexity of $F_k$ is of great importance (see \cite[Remark 2.7]{DemaillyMSRI} for the necessity of convexity). Adding a small Hermitian metric to $F_k$ and taking the dual, the resulting Finsler metric on $E$ is what we want (see Section \ref{sec thm1} for details). 

Whether the assumption in Theorem \ref{thm dual} can be deduced from ampleness of $E$ is still unknown (of course, an affirmative answer would solve Kobayashi's conjecture). Nonetheless, we emphasize that the ampleness of $E$ does imply $S^kE^*$ has a Griffiths negative Hermitian metric $\mathcal{H}_k^*$ for $k$ large; such a Hermitian metric is the dual of  $\mathcal{H}_k$ on $S^kE$ constructed as in (\ref{L2 metric}) except one works on the dual side $O_{P(E^*)}(1)$ (this is the well-known positivity of direct image bundles, see \cite{Berndtsson09,MourouganeTaka, Takayama, positivityandvanishingthmliu,liu2014curvatures}).  However the $k$-th root of $\mathcal{H}^*_k$ is not convex (see Section \ref{sec convex} for examples). This lack of convexity is why we work on $O_{P(E)}(1)$ and use negativity of direct image bundles (see also \cite[Section 3]{berndtsson2009positivity}). Another natural approach is to take the $k$-th root of $\mathcal{H}_k$. The resulting Finsler metric on $E$ is convex by Minkowski's inequality, but it is unclear if this metric is Kobayashi positive. (One can perhaps utilize the intimate connection between the optimal $L^2$ estimates/extensions and the positivity of curvature e.g. \cite{GuanZhou,BoLem,Hacon,deng2018new,ZhouZhu}).

One case in which the assumption in Theorem \ref{thm dual} holds is when 
$E$ has a strongly pseudoconvex Finsler metric with positive Kobayashi curvature. Indeed, such a Finsler metric induces a Hermitian metric $h$ on $O_{P(E)}(1)$ whose curvature has signature $(r-1,n)$ (see Lemma \ref{K positive}). Then following the observation in \cite[Section 4.2]{positivityandvanishingthmliu}, the curvature of $H_k$ in (\ref{L2 metric}) is Griffiths negative for $k$ large. We therefore have 

\begin{corollary}\label{cor}
If $E$ has a strongly pseudoconvex Finsler metric with positive Kobayashi curvature, then $E$ admits a convex and strongly pseudoconvex Finsler metric with positive Kobayashi curvature.
\end{corollary}

Assuming Kobayashi's conjecture is true, Corollary \ref{cor} could be seen as a small step towards Griffiths' conjecture in the sense that the Finsler metric can be made convex, namely a norm, on each fiber (even if not a norm induced by an inner product). 

Using Theorem \ref{thm dual} and Corollary \ref{cor}, we can give a different proof for the known direction in Kobayashi's conjecture (\cite{feng2018complex}). Indeed, given a strongly pseudoconvex Kobayashi positive Finsler metric on $E$, we can construct a Finsler metric on $E^*$ which is strongly plurisubharmonic on $E^*\setminus \{\text{zero section}\}$, so the induced metric on $O_{P(E^*)}(-1)$ has negative curvature, thus $E$ is ample.

We will go over the basics on Kobayashi curvature of Finsler metrics in Section \ref{sec review}. Theorem \ref{thm dual} and Corollary \ref{cor} are proved in Section \ref{sec thm1}. In Section \ref{sec convex}, we give explicit examples about convexity of inner products.
In the last section, we give an extremal characterization of Kobayashi curvature following \cite{LLmax,LLextrapolation}, where Lempert defined curvature for a much broader class of metrics. 

I am indebted to L\'aszl\'o Lempert for many critical comments and suggestions. Part of the paper was done at Purdue University, and I would like to thank the support by NSF grant DMS-1764167.
I am grateful to I-Hsun Tsai and Xiaokui Yang for their interest in this work. Thanks are also due to Academia Sinica and National Center for Theoretical Sciences for providing a stimulating environment.

\section{Kobayashi curvature of Finsler metrics}\label{sec review}

There is a sizable literature for Finsler metrics on vector bundles, for example \cite{Negfinsler, ComplexFinsler, Aikoupartial,AikouMSRI,caowong,liuchern,liudonaldson}. Therefore, we will be very brief on the review; one can consult the above articles for details. We will use \cite{ComplexFinsler}, an updated version of \cite{Negfinsler}. 
\subsection{Finsler metrics}
Let $E$ be a holomorphic vector bundle of rank $r$ over a compact complex manifold $X$ of dimension $n$. For a vector $\zeta\in E_z$, we will symbolically write $(z,\zeta)\in E$. A smooth Finsler metric $F$ on the vector bundle $E\to X$ is a real-valued function on $E$ such that 
\begin{align*}
    &(1) \text{ $F$ is continuous on $E$ and smooth away from the zero section of $E$}.\\
    &(2) \text{ For $(z,\zeta)\in E$}, F(z,\zeta)\geq 0, \text{ and equality holds if and only if $\zeta=0$}.\\
    &(3) \text{ $F(z,\lambda\zeta)=|\lambda|F(z,\zeta)$,  for $\lambda\in \mathbb{C}$}.
\end{align*}
Since $F^2$ will be used more often than $F$ in our computations, we denote $F^2$ by $G$ throughout the paper. 

Denote by $P(E)$ the projectivized bundle of $E$, and by $O_{P(E)}(-1)$ the tautological line bundle over $P(E)$. Let $p$ be the projection from $P(E)$ to $X$, and denote the pull-back bundle $p^*E$ by $\tilde{E}$. In summary,
\[
 \begin{tikzcd}
  O_{P(E)}(-1) \subset  \tilde{E}   \arrow{d}\arrow{r}{\tilde{p}} & E \arrow{d}\\
   P(E) \arrow{r}{p} & X.
  \end{tikzcd}
  \]
For a vector $\zeta\in E_z$, we denote by $[\zeta]$ the equivalence class of $\zeta$ in $P(E_z)$, and we will write $(z,[\zeta])\in P(E)$. The pull-back bundle $\tilde{E}=\coprod_{(z,[\zeta])\in P(E)}(z,[\zeta])\times E_z$, and we will denote an element in $\tilde{E}$ by $(z,[\zeta],Z)$ with $Z\in E_z$. There is a one-to-one correspondence between smooth Finsler metrics on $E$ and smooth Hermitian metrics on $O_{P(E)}(-1)$, furnished by $\tilde{p}$.
If $F$ is a Finsler metric on $E$ with $G=F^2$ and we denote the corresponding Hermitian metric on $O_{P(E)}(-1)$ by $h_G$, then the correspondence is 
\begin{equation}\label{corres}
    h_G(z,[\zeta],\zeta)=G(z,\zeta).
\end{equation}

We use $(z_1,...,z_n)$ for local coordinates on $X$, and $(\zeta_1,...,\zeta_r)$ for local fiber coordinates on $E$ with respect to a holomorphic frame $\{s_1,...,s_r\}$. So $(z_1,...,z_n,\zeta_1,...,\zeta_r)$ is a coordinate system on $E$. Let $F$ be a Finsler metric on $E$ with $G=F^2$. We write 
\begin{align*}
G_i=\partial G/\partial \zeta_i\,,\text{      }\text{  }G_{\bar{j}}=\partial G/\partial \bar{\zeta}_j\,,\text{      }\text{  } G_{i\Bar{j}}=\partial^2 G/\partial \zeta_i \partial\bar{\zeta}_{j}\,, \\    G_{i\alpha}=\partial G_i/\partial z_\alpha\,,\text{      }\text{  } G_{i\bar{j}\bar{\beta}}=\partial G_{i\bar{j}}/\partial \bar{z}_\beta\,,\text{ etc.,}
\end{align*}
with Latin letters $i,j$ for the fiber direction and Greek letters $\alpha,\beta$ for the base direction. The following formulas will be used later; their derivation can be found in \cite[p. 147]{ComplexFinsler}. For $\lambda\in \mathbb{C}$, 
\begin{align}
&G_{k\bar{j}}(z,\lambda\zeta)=G_{k\bar{j}}(z,\zeta)\label{6}.\\
&G(z,\zeta)=\sum_i G_i(z,\zeta)\zeta_i=\sum_{i,j}G_{i\bar{j}}(z,\zeta)\zeta_i\bar{\zeta}_j\label{7}.
\end{align}

\subsection{Strongly pseudoconvex Finsler metrics}

A Finsler metric $F$ is said to be
\begin{align*}
(1) &\text{ strongly pseudoconvex if $(F_{i\bar{j}})$ is positive definite on $E\setminus \{\text{zero section}\}$}.\\
(2) &\text{ convex if $F$ restricted to each fiber $E_z$ is convex},\\ &\text{namely the fiberwise real Hessian of $F$ is positive semidefinite on $E\setminus \{\text{zero section}\}$}.\\
(3) &\text{ strongly convex if the fiberwise real Hessian of $G$ is positive on $E\setminus \{\text{zero section}\}$}.
\end{align*}
Note that in (1), the requirement $(F_{i\bar{j}})>0$ is equivalent to $(G_{i\bar{j}})>0$, which can be shown by using the formulas in \cite[(3.3) and (3.10)]{ComplexFinsler}. A Hermitian metric always satisfies (1), (2), and (3). Also notice that a 'convex' Finsler metric in \cite{Negfinsler} is what we call 'strongly pseudoconvex' here. Let us recall a lemma of strong pseudoconvexity (\cite[p. 150 Theorem (1)]{ComplexFinsler}).
\begin{lemma}\label{fiber neg}
 Let $F$ be a Finsler metric on $E$ and $h_G$ be its corresponding Hermitian metric on $O_{P(E)}(-1)$. Then $F$ is strongly pseudoconvex if and only if the curvature $\Theta(h_G)$ of $h_G$ restricted to each fiber of $P(E)\to X$ is negative definite.
\end{lemma}

In the rest of this subsection, we will assume the Finsler metric $F$ is strongly pseudoconvex. There is a natural Hermitian metric $\tilde{G}$ on the pull-back bundle $\tilde{E}$ (see \cite{Negfinsler}). In terms of local coordinates, the Hermitian metric $\tilde{G}$ is given by  
$$\tilde{G}_{(z,[\zeta])}(Z,Z)=\sum_{i,j}G_{i\bar{j}}(z,\zeta)Z_i\bar{Z}_j, \text{ for $Z=\sum^r_{i=1}Z_i s_i(z)\in E_z$}.$$
By (\ref{7}), the metric $\tilde{G}$ restricted to $O_{P(E)}(-1)$ is the same as the Hermitian metric $h_G$ on $O_{P(E)}(-1)$ coming from $G$ on $E$ through the one-to-one correspondence (\ref{corres}). 

Now $(\tilde{E},\tilde{G})$ is a Hermitian holomorphic vector bundle, so we can talk about its Chern curvature $\Theta$, an $\End \tilde{E}$-valued $(1,1)$-form on $P(E)$. With respect to the metric $\tilde{G}$, the bundle $\tilde{E}$ has a fiberwise orthogonal decomposition $O_{P(E)}(-1)\oplus O_{P(E)}(-1)^\perp$, and so $\Theta$ can be written as a block matrix. Let $\Theta|_{O_{P(E)}(-1)}$ denote the block in the matrix $\Theta$ corresponding to $\End(O_{P(E)}(-1))$. Since $O_{P(E)}(-1)$ is a line bundle, $\Theta|_{O_{P(E)}(-1)}$ is a $(1,1)$-form on $P(E)$.
\begin{definition}
The $(1,1)$-form $\Theta|_{O_{P(E)}(-1)}$ is called the Kobayashi curvature of the strongly pseudoconvex Finsler metric $F$. Kobayashi positivity (or negativity) means the positivity (or negativity) of $\Theta|_{O_{P(E)}(-1)}$.
\end{definition}

Following the computation in \cite[Section 5]{ComplexFinsler}, we have 
\begin{lemma}\label{K positive}
Let $F$ be a strongly pseudoconvex Finsler metric on $E$ and $h_G$ be the corresponding Hermitian metric on $O_{P(E)}(-1)$. Then $F$ is Kobayashi positive (or negative) if and only if $\Theta(h_G)$ has signature $(n,r-1)$ (or $(0,n+r-1)$).
\end{lemma}

We give a local expression of the Kobayashi curvature that will be used in the final section. We focus on a local chart $\{(z,[\zeta]):\zeta_r\neq 0\}$ of $P(E)$, and let $\zeta_i/\zeta_r=w_i$ for $1\leq i\leq  r-1$. The curvature $\Theta$ of $\tilde{G}$
can be locally written as
\begin{equation*}
    \Theta=\sum_{\alpha,\beta}R_{\alpha\bar{\beta}}\,dz_\alpha\wedge d\bar{z}_\beta+\sum_{\alpha,l}P_{\alpha\bar{l}}\,dz_\alpha\wedge d\bar{w}_l+\sum_{k,\beta}\mathcal{P}_{k\bar{\beta}}\,dw_k\wedge d\bar{z}_\beta+\sum_{k,l}Q_{k\bar{l}}\,dw_k\wedge d\bar{w}_l,
\end{equation*}
where $R_{\alpha\bar{\beta}},P_{\alpha\bar{l}},\mathcal{P}_{k\bar{\beta}}$, and $Q_{k\bar{l}}$ are endomorphisms of $\tilde{E}$. Since $O_{P(E)}(-1)$ is a line bundle,
\begin{align*}
    \Theta|_{O_{P(E)}(-1)}&=\sum_{\alpha,\beta} \frac{\tilde{G}(R_{\alpha\bar{\beta}}\zeta,\zeta)}{\tilde{G}(\zeta,\zeta)} \,dz_\alpha\wedge d\bar{z}_\beta+\sum_{\alpha,l}\frac{\tilde{G}(P_{\alpha\bar{l}}\zeta,\zeta)}{\tilde{G}(\zeta,\zeta)}\,dz_\alpha\wedge d\bar{w}_l\\&+\sum_{k,\beta}\frac{\tilde{G}(\mathcal{P}_{k\bar{\beta}}\zeta,\zeta)}{\tilde{G}(\zeta,\zeta)}\,dw_k\wedge d\bar{z}_\beta+\sum_{k,l}\frac{\tilde{G}(Q_{k\bar{l}}\zeta,\zeta)}{\tilde{G}(\zeta,\zeta)}\,dw_k\wedge d\bar{w}_l,
\end{align*}
and by \cite[Formula (5.8)]{ComplexFinsler} the last three terms vanish, therefore
\begin{equation}\label{local for koba}
  \Theta|_{O_{P(E)}(-1)}=\sum_{\alpha,\beta} \frac{\tilde{G}(R_{\alpha\bar{\beta}}\zeta,\zeta)}{\tilde{G}(\zeta,\zeta)}\,dz_\alpha \wedge d\bar{z}_\beta.  
\end{equation}

\section{Proofs of Theorem \ref{thm dual} and Corollary \ref{cor}}\label{sec thm1}
Recall that the $L^2$-metric $H_k$ on $S^kE^*$ is 
\begin{align}\label{L2 metric in 3}
H_k(u,v):=\int_{P(E_z)} h^k(\Phi_{k,z}(u),\Phi_{k,z}(v) ) \omega_z^{r-1}, \textup{ for $u \textup{ and } v\in S^kE^*_z$,}
\end{align}
where $h$ is a Hermitian metric on $O_{P(E)}(1)$ and $\{\omega_z\}_{z\in X}$ is a family of Hermitian metrics on the fibers $P(E_z)$, and  $\Phi_{k,z}$ is the canonical isomorphism from $S^kE^*_z$ to $ H^0(P(E_z),O_{P(E_z)}(k))$.   

\begin{proof}[Proof of Theorem \ref{thm dual}]
According to the assumption, the Hermitian metric $H_k$ is Griffiths negative for some $k$; equivalently, $\log H_k(s,s)$ is strongly plurisubharmonic for any nonvanishing holomorphic section $s$ of $S^kE^*$. 

If for a vector $v\in E^*$, we denote its $k$-th tensor power by $v^k\in S^k E^*$, then one can define a Finsler metric $F_k$ on $E^*$ by setting $F_k(v)=H_k(v^k,v^k)^{1/2k}$. By Griffiths negativity of $H_k$, $F_k$ is strongly plurisubharmonic on the total space $E^*\setminus \{\text{zero section}\}$.

Next, we will use the $L^2$-integral feature of $H_k$ to prove that the Finsler metric $F_k$ is convex i.e., it is convex on each fiber $E^*_z$. Fix $z\in X$ and two vectors $u,v\in E^*_z$. Let $\{s_1,...,s_r\}$ be a holomorphic frame of $E$ around $z$ with fiber coordinates $(\zeta_1,...,\zeta_r)$, and $\{t_1,...,t_r\}$ be the dual frame on $E^*$. The frame $\{t_i\}$ on $E^*$ induces a frame $\{t^\alpha\}$ on $S^kE^*$ with $t^\alpha=t_1^{\alpha_1}\cdot\cdot\cdot t_r^{\alpha_r}$ and $\alpha_1+...+\alpha_r=k$. On $P(E_z)$, we consider the chart $\{[\zeta_1,...,\zeta_r]: \zeta_r\neq 0\}$, then $$e:=\frac{\sum_i \zeta_i s_i}{\zeta_r}$$ is a holomorphic frame of $O_{P(E_z)}(-1)$ with dual frame $e^*$ for $O_{P(E_z)}(1)$ and frame $(e^*)^k$ for  $O_{P(E_z)}(k)$. The global section $
 \Phi_{k,z}(t^\alpha)\in  H^0(P(E_z),O_{P(E_z)}(k))$ can be locally written as $$\frac{\zeta^\alpha}{\zeta_r^k}(e^*)^k, \text{where $\zeta^\alpha=\zeta_1^{\alpha_1}\cdot\cdot\cdot \zeta_r^{\alpha_r}$.
}
$$
Therefore, if $E^*_z\ni u=\sum_i u_it_i$ and $v=\sum_i v_i t_i$ with $u_i,v_i\in \mathbb{C}$, then 
\begin{equation}\label{u+v}
    \begin{aligned}
    h^k\big( \Phi_{k,z}((u+v)^k),\Phi_{k,z}((u+v)^k)  \big)=  \frac{\bigr|\sum_i (u_i+v_i) \zeta_i\bigr|^{2k}}{|\zeta_r|^{2k}}h(e^*,e^*)^k.
  \end{aligned}
\end{equation}
If we denote by $h^*$ the dual metric on $O_{P(E)}(-1)$, then $h(e^*,e^*)=|\zeta_r|^2/h^*(\sum_i\zeta_i s_i,\sum_i\zeta_i s_i)$. So (\ref{u+v}) becomes 
\begin{align}\label{global}
    h^k\big( \Phi_{k,z}((u+v)^k),\Phi_{k,z}((u+v)^k)  \big)= \frac{\bigr|\sum_i (u_i+v_i) \zeta_i\bigr|^{2k}}{h^*(\sum_i\zeta_i s_i,\sum_i\zeta_i s_i)^k};
\end{align}
the right hand side of formula (\ref{global}) is meaningful on the entire $P(E_z)$, and we will abbreviate the denominator by $(h^*)^k$ in the computation below.

Using formula (\ref{global}), we have
\begin{align*}
  F_k(u+v)&=H_k\big((u+v)^k,(u+v)^k\big)^{1/2k}\\
  &=\bigr(\int_{P(E_z)}h^k\big( \Phi_{k,z}((u+v)^k),\Phi_{k,z}((u+v)^k)  \big) \omega_z^{r-1} \bigr)^{1/2k}\\
  &=\bigr(\int_{P(E_z)}\frac{ \bigr|\sum_i (u_i+v_i) \zeta_i\bigr|^{2k}}{(h^*)^k}  \omega_z^{r-1} \bigr)^{1/2k}\\ 
  &\leq \bigr(\int_{P(E_z)} \big( \frac{\bigr|\sum_i u_i\zeta_i\bigr|}{(h^*)^{1/2}}    + \frac{\bigr|\sum_i v_i \zeta_i\bigr|}{(h^*)^{1/2}}  \big) ^{2k}  \omega_z^{r-1} \bigr)^{1/2k}\\
  &\leq
  \bigr(\int_{P(E_z)} \frac{\bigr|\sum_iu_i\zeta_i\bigr|^{2k}}{(h^*)^k}\omega_z^{r-1}
  \bigr)^{1/2k}+
  \bigr(\int_{P(E_z)}\frac{ \bigr|\sum_iv_i\zeta_i\bigr|^{2k}}{(h^*)^k}\omega_z^{r-1}
  \bigr)^{1/2k}\\
  &=H_k(u^k,u^k)^{1/2k}+H_k(v^k,v^k)^{1/2k}
  \\
  &=F_k(u)+F_k(v),
\end{align*}
where the second inequality is due to the Minkowski inequality. We have just shown that $F_k(u+v)\leq F_k(u)+F_k(v)$, and this implies convexity $F_k(\theta u+(1-\theta) v)\leq \theta F_k(u)+(1-\theta)F_k(v)$ for $0\leq \theta \leq 1$. So $F_k$ is a convex Finsler metric which is strongly plurisubharmonic on $E^*\setminus \{\text{zero section}\}$.

For the second part of the theorem, since the Finsler metric $F_k$ from above is convex and strongly plurisubharmonic, after adding a small Hermitian metric we obtain a strongly convex and strongly plurisubharmonic Finsler metric on $E^*$, whose dual metric has transversal Levi signature $(r,n)$ according to \cite[Theorem 2.5]{DemaillyMSRI} and  \cite{Sommese} (see below for the definition of transversal Levi signature). By Lemma \ref{lem signature} below we obtain a convex and strongly pseudoconvex Kobayashi positive Finsler metric on $E$.   
\end{proof}

\begin{lemma}\label{lem signature}
 If a Finsler metric $F$ has tranversal Levi signature $(r,n)$, then $F$ is strongly pseudoconvex and has positive Kobayashi curvature.
\end{lemma}
\begin{proof}
This lemma has appeared in Proposition 1.8 in the first version of  \cite{feng2018complex}. We give a simplified proof. Recall that a Finsler metric $F$ is said to have transversal Levi signature $(r,n)$, if at every point $(z,\zeta)\in E\setminus \{\text{zero section}\}$, the Levi form $i\partial\bar{\partial}F$ on $E$ is positive definite along the fiber $E_z$ and negative definite on some $n$-dimensional subspace $W\subset T^{1,0}_{(z,\zeta)}E$ which is transversal to the fiber $E_z$. 

Since $F_{i\bar{j}}>0$, the metric $F$ is strongly pseudoconvex. Let $\pi:E\setminus\{\text{zero section}\}\to P(E)$ be the quotient map. It is not hard to verify $\pi^*\Theta(h_G)=-\partial \bar{\partial }\log G$. For any nonzero $e\in W$,
\begin{align*}
 \Theta(h_G)(\pi_*e,\overline{\pi_*e})=\pi^*\Theta(h_G)(e,\overline{e})
 =-2\partial \bar{\partial }\log F(e,\overline{e})\\
 =-2(\frac{\partial \bar{\partial }F}{F}-\frac{\partial F\wedge \bar{\partial}F }{F^2})(e,\overline{e})\geq -2\frac{\partial \bar{\partial }F}{F}(e,\overline{e})>0. 
\end{align*}
As a consequence, $\pi_*W$ is of dimension $n$, and $\Theta(h_G)$ is positive definite on $\pi_*W$. By Lemma \ref{fiber neg} and a dimension count, $\Theta(h_G)$ has signature $(n,r-1)$, thus $F$ is Kobayashi positive by Lemma \ref{K positive}.
\end{proof}

\begin{proof}[Proof of Corollary \ref{cor}]

By Lemmas \ref{fiber neg} and \ref{K positive}, a strongly pseudoconvex Kobayashi positive Finsler metric $F$ on $E$ induces a Hermitian metric $h$ on $O_{P(E)}(1)$ such that the curvature $\Theta(h)$ restricted to each fiber $P(E_z)$ is positive, and $\Theta(h)$ has signature $(r-1,n)$. With this $h$ on $O_{P(E)}(1)$ and fixing a family of Hermitian metrics $\{\omega_z\}_{z\in X}$ on the fibers $P(E_z)$, we have the Hermitian metric $H_k$ in (\ref{L2 metric}) on $S^kE^*$.

Given a point $(z_0,\zeta_0)\in E$, we can find a frame $\{s_i\}$ of $E$ such that 
$G_{i\bar{j}}(z_0,\zeta_0)=\delta_{ij}$ and $G_{i\bar{j}\alpha}(z_0,\zeta_0)=G_{i\bar{j}\bar{\beta}}(z_0,\zeta_0)=0$ (such a frame can be obtained by (5.11) in \cite{ComplexFinsler}). In this frame, $$\Theta(h)|_{(z_0,[\zeta_0])}=-\sum_{\alpha,\beta}\frac{\partial ^2\log h}{\partial z_\alpha \partial \bar{z}_\beta} dz_\alpha \wedge d\bar{z}_\beta-\sum_{i,j}\frac{\partial ^2\log h}{\partial w_i \partial \bar{w}_j} dw_i \wedge d\bar{w}_j.$$
Since $\Theta(h)|_{P(E_{z_0})}$ is positive and $\Theta(h)$ has signature $(r-1,n)$, the matrix $$\bigr(-\frac{\partial ^2\log h}{\partial z_\alpha \partial \bar{z}_\beta}\bigr)$$
is negative at $(z_0,[\zeta_0])$ hence negative in a neighborhood of $(z_0,[\zeta_0])$ in $P(E)$.

Therefore, $P(E)$ can be covered by a finite number of open sets $\{U_m\}$; each $U_m$ corresponds to a frame of $E$ such that the corresponding matrix $(-\partial ^2\log h/\partial z_\alpha \partial \bar{z}_\beta)$ is negative in $U_m$. For a fixed $U_m$, we write $$\omega_z=\sqrt{-1}\sum_{i,j}g_{i\bar{j}}(z,w)dw_i\wedge d\bar{w}_j,$$ so $\omega_z^{r-1}=\det (g_{i\bar{j}})\bigwedge_l(\sqrt{-1}dw_l\wedge d\bar{w}_l) $ that we will abbreviate as $\det(g) dw\wedge d\bar{w}$. By taking $k$ large, the following matrix can be made negative 
\begin{equation}\label{k large}
\big(\frac{\partial^2 -k\log h-\log \det(g)}{\partial z_\alpha\partial \bar{z}_\beta}\big).
\end{equation}
Since there are only finitely many $U_m$, we can take $k$ large such that the corresponding (\ref{k large}) in each $U_m$ is negative. 

For a fixed point $z_0\in X$, we partition $P(E_{z_0})=\bigcup_l V_l$ with $V_l$ in $U_m$ for some $m$. Let us denote the curvature of $H_k$ by $\Theta$ an $\End S^kE^*$-valued (1,1)-form. Given a nonzero $u\in S^kE^*_{z_0}$, we would like to show the (1,1)-form $H_k(\Theta u, u)$ is negative, namely, $H_k(\Theta u,u) (\eta,\bar{\eta})<0$ for $0\neq \eta\in T^{1,0}_{z_0}X$.  First, we extend $u$ to a local holomorphic section whose covariant derivative (with respect to $H_k$) at $z_0$ is zero. We still denote the holomorphic extension by $u$. A simple computation shows, at $z_0$, \begin{equation}\partial\label{right} \bar{\partial}H_k(u,u) (\eta,\bar{\eta})=-H_k(\Theta u,u) (\eta,\bar{\eta}).\end{equation} 
Meanwhile, if we denote $p$ as the projection from $P(E)$ to $X$, then by definition 
\begin{equation}\label{left}
\begin{aligned}
  \partial \bar{\partial}H_k(u,u)&=\partial \bar{\partial}p_*\big(h^k(\Phi_{k,z}(u),\Phi_{k,z}(u) ) \omega_z^{r-1}\big)=p_*\partial \bar{\partial}\big(h^k(\Phi_{k,z}(u),\Phi_{k,z}(u) ) \omega_z^{r-1}\big)\\
  &=\sum_l \int_{V_l}\partial \bar{\partial}\big(h^k(\Phi_{k,z}(u),\Phi_{k,z}(u) ) \omega_z^{r-1}\big).
\end{aligned}
\end{equation}
Let us compute $\partial \bar{\partial}\big(h^k(\Phi_{k,z}(u),\Phi_{k,z}(u) ) \omega_z^{r-1}\big)$ in $U_m$. In $U_m$, we can write $\Phi_{k,z}(u)$ as $fe$ with $f$ a scalar-valued holomorphic function and $e$ a frame for $O_{P(E)}(k)$. So, $h^k(\Phi_{k,z}(u),\Phi_{k,z}(u) )=|f|^2h^k(e,e)$, and we will abbreviate $h^k(e,e)$ by $h^k$. Moreover, we have $\omega_z^{r-1}= \det(g) dw\wedge d\bar{w}$. Therefore, in $U_m$,  \begin{align*}
\partial \bar{\partial}\big(h^k(\Phi_{k,z}(u),\Phi_{k,z}(u) ) \omega_z^{r-1}\big)&= \partial \bar{\partial}\big(|f|^2h^k \det(g) dw\wedge d\bar{w}\big)\\
&=\sum_{\alpha,\beta} \frac{\partial^2 |f|^2h^k\det(g)}{\partial z_\alpha \partial\bar{z}_\beta}dz_\alpha\wedge d\bar{z}_\beta\wedge dw\wedge d\bar{w}.
\end{align*}   
Let us denote $h^k\det (g)$ by $e^{-\phi}$ temporarily. Let $\eta=\sum_{\alpha}\eta_\alpha \partial/\partial z_\alpha\in T^{1,0}_{z_0}X$. A straightforward computation shows 
\begin{equation}\label{curvature}
    \sum_{\alpha,\beta}\frac{\partial^2 |f|^2 e^{-\phi}}{\partial z_\alpha \partial\bar{z}_\beta}\eta_\alpha \bar{\eta}_\beta=\big( 
    e^{-\phi}\bigr|\sum_\alpha \frac{\partial f}{\partial z_\alpha}\eta_\alpha-f\sum_\alpha \eta_\alpha \frac{\partial \phi}{\partial z_\alpha}\bigr|^2-|f|^2e^{-\phi}\sum_{\alpha,\beta}\frac{\partial^2 \phi}{\partial z_\alpha\partial \bar{z}_\beta}\eta_\alpha \bar{\eta}_\beta
    \big).
\end{equation}
The matrix in the last term of (\ref{curvature})
$$\frac{\partial^2 \phi}{\partial z_\alpha\partial \bar{z}_\beta}=\frac{\partial^2 -k\log h-\log \det(g)}{\partial z_\alpha\partial \bar{z}_\beta}$$
is negative by (\ref{k large}), so (\ref{curvature}) is positive, and 
\begin{equation}\label{vm}
  \int_{V_l}\partial \bar{\partial}\big(h^k(\Phi_{k,z}(u),\Phi_{k,z}(u) ) \omega_z^{r-1}\big)(\eta,\bar{\eta})>0.  
\end{equation}
Hence by (\ref{right}) and (\ref{left}), $H_k(\Theta u,u)(\eta,\bar{\eta})<0$. Therefore, the Hermitian metric $H_k$ on $S^kE^*$ is Griffith negative for $k$ large. The proof is complete by using Theorem \ref{thm dual}.
 
The above curvature computation is a combination of ideas from \cite[Section 4.2]{positivityandvanishingthmliu}, \cite[Section 3]{berndtsson2009positivity}.
\end{proof}

\section{Convexity of the $k$-th root of an inner product}\label{sec convex}

Given a complex vector space $V$ of dimension $r$, let $S^kV$ be the $k$-th symmetric power of $V$. For an inner product $H$ on $S^kV$, the map $V\ni v\mapsto H(v^k,v^k)$ is in general not convex, and $V\ni v\mapsto H(v^k,v^k)^{1/2k}$ is even less so.

\begin{example}
Take $r=2$ and a basis $\{e_1,e_2\}$ for $V$, then $\{e_1^2,e_2^2,e_1e_2\}$ is a basis for $S^2V$. Consider an inner product $H$ on $S^2V$ with matrix representation under $\{e_1^2,e_2^2,e_1e_2\}$ equal to 
\[
\begin{pmatrix}
1/2 & 0 & 0\\
0 & 1/2 & 0\\
0&0&1
\end{pmatrix}.
\] We write $v\in V$ as $v=v_1e_1+v_2e_2$ with $v_1=x_1+\sqrt{-1}y_1$, $v_2=x_2+\sqrt{-1}y_2$, and $x_j,y_j$ real. So $$H(v^2,v^2)= |v_1^2|^2/2+ |v_2^2|^2/2+|2v_1v_2|^2= (x_1^2+y_1^2)^2/2+ (x_2^2+y_2^2)^2/2+4(x_1^2+y_1^2)(x_2^2+y_2^2),$$ 
\[
\begin{pmatrix}\frac{\partial^2 H}{\partial x_i \partial x_j}\end{pmatrix}_{i,j}=
\begin{pmatrix}
4x_1^2+2(x_1^2+y_1^2)   +8(x_2^2+y_2^2) & 16x_1x_2\\
16x_1x_2 & 4x_2^2+2(x_2^2+y_2^2)  +8(x_1^2+y_1^2)
\end{pmatrix}.
\]
At $x_1=x_2=1$ and $y_1=y_2=0$, this 2-by-2 matrix equals 
\[
\begin{pmatrix}
14 & 16\\
16 & 14
\end{pmatrix}
\] which is not semipositive. So the map $v\mapsto H(v^2,v^2)$ is not convex.
\end{example}

The next example shows that if the $k$-th root of an inner product is convex, the $k$-th root of the dual inner product is in general not convex.

\begin{example}
Let $V$ be a complex vector space of dimension two with basis $\{e_1,e_2\}$. Then $\{e_1^2,e_2^2,e_1e_2\}$ is a basis for $S^2V$. Let $\{e_1^*,e_2^*\}$ be the dual basis for $V^*$, and we use $(\zeta_1,\zeta_2)$ to denote vectors in $V^*$. Let $h$ be an inner product on $V^*$ such that $\{e_1^*,e_2^*\}$ is orthonormal. This $h$ induces on $O_{P(V
^*)}(1)$ a Hermitian metric which we still denote by $h$; the curvature $\omega$ of $h$ is a K\"ahler form on $P(V^*)$. Through the isomorphism $\Phi:S^2V\to H^0(P(V^*),O_{P(V^*)}(2))$, we define an inner product $H$ on $S^2 V$ by, for $\xi,\eta\in S^2V$, 
$$
H(\xi,\eta)=\int_{P(V^*)}h^2(\Phi(\xi),\Phi(\eta))\omega.
$$
On the chart $\{[\zeta_1,\zeta_2]:\zeta_2\neq 0\}$ of $P(V^*)$, we have a local frame $$e:=\frac{\zeta_1}{\zeta_2}e_1^*+e_2^*$$ for $O_{P(V^*)}(-1)$, and a dual frame $e^*$ for $O_{P(V^*)}(1)$. Therefore, the global sections have local expressions  
\begin{align}
    \Phi(e_1^2)=\frac{\zeta_1^2}{\zeta_2^2}(e^*)^2,\text{ } \Phi(e_2^2)=(e^*)^2,\text{ and } \Phi(e_1e_2)=\frac{\zeta_1}{\zeta_2}(e^*)^2.
\end{align}
Moreover, if we denote $\zeta_1/\zeta_2$ by $z$, then $h(e^*,e^*)=1/(|z|^2+1)$ and $\omega=dz\wedge d\bar{z}/(|z|^2+1)^2$. Having these information, one can compute the matrix representation of $H$ in terms of $\{e_1^2,e_2^2,e_1e_2\}$. For example, 
\begin{align*}
  H(e_1^2,e_2^2)=\int_{\mathbb{C}}z^2\frac{1}{(|z|^2+1)^2}\frac{dz\wedge d\bar{z}}{(|z|^2+1)^2}=0,\\ 
  H(e_1^2,e_1^2)=\int_{\mathbb{C}}|z|^4\frac{1}{(|z|^2+1)^2}\frac{dz\wedge d\bar{z}}{(|z|^2+1)^2}=\frac{\pi}{3},
\end{align*}
and similarly $H(e_2^2,e_2^2)=\pi/3, H(e_1e_2,e_1e_2)=\pi/6, H(e_1^2,e_1e_2)=H(e_2^2,e_1e_2)=0$. Therefore, the matrix representation under $\{e_1^2,e_2^2,e_1e_2\}$ is
\[
\begin{pmatrix}
\frac{\pi}{3}&0&0\\
0&\frac{\pi}{3}&0\\
0&0&\frac{\pi}{6}
\end{pmatrix}.
\]
The convexity of the map $V\ni v\mapsto H(v^2,v^2)$ follows by the Minkowski inequality as in the proof of Theorem \ref{thm dual}. But the dual metric $H^*$ has matrix representation under $\{{e^*_1}^2,{e^*_2}^2,e^*_1e^*_2\}$
\[
\frac{6}{\pi}\begin{pmatrix}
1/2&0&0\\
0&1/2&0\\
0&0&1
\end{pmatrix},
\]
and we know from the first example that $V^*\ni v\mapsto H^*(v^2,v^2)$ is not convex. 
\end{example}

\section{Extremal characterization of the Kobayashi curvature}\label{sec extre}

In this section we discuss a different way of obtaining the Kobayashi curvature by following a curvature notion defined in \cite{LLmax,LLextrapolation} (see also \cite{rochberg1984}). Finsler metrics are assumed convex in \cite{LLmax,LLextrapolation}, while in our case we assume strong pseudoconvexity. This characterization has the potential to define curvature for Finsler metrics with weaker convexity.

Let $F$ be a strongly pseudoconvex Finsler metric on $E$ with $F^2=G$. For $z_0\in X$, $v\in T^{1,0}_{z_0} X$ and $0\neq\zeta\in E_{z_0}$, we define
\begin{equation}\label{inf}
K_v(\zeta)=-\inf \partial\bar{\partial}\log G(\phi)\bigr|_{z_0}(v,\bar{v}),
\end{equation}
the inf taken over local holomorphic sections $\phi$ of $E$ such that $\phi(z_0)=\zeta$. 

Consider a tangent vector $\tilde{v}$  to $P(E)$ at $(z_0,[\zeta])$ such that $p_*(\tilde{v})=v$. By (\ref{local for koba}), $\Theta|_{O_{P(E)}(-1)}(\tilde{v},\bar{\tilde{v}})$ is independent of the lift $\tilde{v}$, so we will simply write $\Theta|_{O_{P(E)}(-1)}(v,\bar{v})$.

\begin{lemma}\label{lem inf}
 $K_v(\zeta)=\Theta|_{O_{P(E)}(-1)}(v,\bar{v})$. 
\end{lemma}

\begin{proof}
Fix a frame $\{s_1,...,s_r\}$ of $E$ around $z_0$ such that
$G_{i\bar{j}}(z_0,\zeta)=\delta_{ij}$ and $G_{i\bar{j}\alpha}(z_0,\zeta)=G_{i\bar{j}\bar{\beta}}(z_0,\zeta)=0$ (see (5.11) in \cite{ComplexFinsler}). Since the vector $\zeta\in E_{z_0}$ is nonzero, without loss of generality we assume its $r$-th component $\zeta_r\neq 0$. Suppose $v=\sum v_\alpha\partial/\partial z_\alpha$. Under this local trivialization, the holomorphic section $\phi=(\phi_1,...,\phi_r)$ with $\phi_r\neq 0$ near $z_0$. The computation will be done on a chart $\{(z,[w_1,...,w_{r-1},1])\}$ of $P(E)$. The connection form of the Chern connection of $\tilde{G}$ is an $\End \tilde{E}$-valued $(1,0)$-form on $P(E)$ that we write as $\sum_\alpha A_\alpha dz_\alpha+\sum_i B_i dw_i$ with $A_\alpha$ and $B_i\in \End \tilde{E}$. Moreover, $A_\alpha(z_0,[\zeta])=0$ because of the choice of the frame 
$\{s_i\}$.
By setting $D_{z_\alpha}=\partial/\partial z_\alpha+A_\alpha$ and $D_{w_i}=\partial/\partial w_i+B_i$, we have
\begin{align*}
  &\frac{\partial }{\partial z_\alpha}\tilde{G}_{(z,[w,1])}(\phi(z),\phi(z))=\tilde{G}_{(z,[w,1])}(D_{z_\alpha}\phi,\phi),\\
  &\frac{\partial }{\partial w_i}\tilde{G}_{(z,[w,1])}(\phi(z),\phi(z))=\tilde{G}_{(z,[w,1])}(D_{w_i}\phi,\phi).  
\end{align*}

By (\ref{7}), $G(\phi)=G(z,\phi(z))=\tilde{G}_{(z,[\phi(z)])}(\phi(z),\phi(z))$ and we will suppress the subscript $(z,[\phi(z)])$ of $\tilde{G}$ for simplicity. Therefore, by chain rule 
\begin{align*}
\partial\bar{\partial} G(z,\phi(z))(v,\bar{v})&=\sum^n_{\alpha,\beta=1}\frac{\partial^2  }{\partial z_\alpha \partial \bar{z}_\beta}\tilde{G}(\phi(z),\phi(z))v_\alpha \bar{v}_\beta\\
&=\big\|\sum_\alpha v_\alpha D_{z_\alpha}\phi+\sum_{i,\alpha}D_{w_i}\phi\frac{\partial (\frac{\phi_i}{\phi_r})}{\partial z_\alpha}v_\alpha\big\|^2_{\tilde{G}}\\
&-\sum_{\alpha,\beta}v_\alpha \bar{v}_\beta\big[\tilde{G}(R_{\alpha\bar{\beta}}\phi,\phi)+\sum^{r-1}_{i=1}\tilde{G}(P_{\alpha\bar{i}}\phi,\phi)\overline{\frac{\partial (\frac{\phi_i}{\phi_r})}{\partial z_\beta}}\\
&+\sum^{r-1}_{i=1}\tilde{G}(\mathcal{P}_{i\bar{\beta}}\phi,\phi)\frac{\partial (\frac{\phi_i}{\phi_r})}{\partial z_\alpha}+\sum^{r-1}_{i,j=1}\tilde{G}(Q_{i\bar{j}}\phi,\phi)\frac{\partial (\frac{\phi_i}{\phi_r})}{\partial z_\alpha}\overline{\frac{\partial (\frac{\phi_j}{\phi_r})}{\partial z_\beta}}\big].
\end{align*}
The last three terms involving $P,\mathcal{P}$, and $Q$ are zero by \cite[(5.8)]{ComplexFinsler}. Meanwhile, by the Cauchy--Schwarz inequality 
\begin{align*}
    \big| \sum_\alpha \frac{\partial}{\partial z_\alpha} G(z,\phi(z)) v_\alpha \big|^2=&\big|\tilde{G}\big(\sum_\alpha v_\alpha D_{z_\alpha}\phi+\sum_{i,\alpha}D_{w_i}\phi\frac{\partial (\frac{\phi_i}{\phi_r})}{\partial z_\alpha}v_\alpha,\phi\big)\big|^2\\
    \leq& \big\|\sum_\alpha v_\alpha D_{z_\alpha}\phi+\sum_{i,\alpha}D_{w_i}\phi\frac{\partial (\frac{\phi_i}{\phi_r})}{\partial z_\alpha}v_\alpha\big\|^2_{\tilde{G}} \|\phi\|^2_{\tilde{G}}.
\end{align*}
All added up, we see 
\begin{align*}
    &\partial \bar{\partial}\log G(z,\phi(z))(v,\bar{v})\\
    &= \frac{1}{\|\phi\|^4_{\tilde{G}}}   \big(\big\|\sum_\alpha v_\alpha D_{z_\alpha}\phi+\sum_{i,\alpha}D_{w_i}\phi\frac{\partial (\frac{\phi_i}{\phi_r})}{\partial z_\alpha}v_\alpha\big\|^2_{\tilde{G}}  \| \phi \|^2_{\tilde{G}}   -\sum_{\alpha,\beta}v_\alpha \bar{v}_\beta \tilde{G}(R_{\alpha\bar{\beta}}\phi,\phi)\| \phi \|^2_{\tilde{G}}\\ 
    &- \big|\sum_\alpha \frac{\partial}{\partial z_\alpha} G(z,\phi(z)) v_\alpha \big|^2\big)\\
    &\geq -\frac{\sum_{\alpha,\beta}v_\alpha \bar{v}_\beta\tilde{G}( R_{\alpha\bar{\beta}}\phi,\phi)}{\|\phi\|^2_{\tilde{G}}}.
    \end{align*}
Therefore,
\begin{align*}
    \partial \bar{\partial}\log G(z,\phi(z))(v,\bar{v})\geq -\sum_{\alpha,\beta}\frac{\tilde{G} (R_{\alpha\bar{\beta}}\zeta,\zeta)}{\tilde{G}(\zeta,\zeta)}v_\alpha\bar{v}_\beta; 
\end{align*}   
the equality holds if we choose $\phi(z)=(\zeta_1,...,\zeta_r)+ (\zeta_1,...,\zeta_r)\sum_\alpha(z-z_0)_\alpha$ under the frame $\{s_i\}$. Indeed,  $$\frac{\partial (\frac{\phi_i}{\phi_r})}{\partial z_\alpha}\bigr|_{z_0}=\frac{\frac{\partial \phi_i}{\partial z_\alpha}\phi_r-\phi_i \frac{\partial \phi_r}{z_\alpha}   }{\phi_r^2}\bigr|_{z_0}=0, \text{ and }  A_\alpha(z_0,[\zeta])=0. $$   
So, at $(z_0,[\zeta])$,  $$\sum_\alpha v_\alpha D_{z_\alpha}\phi+\sum_{i,\alpha}D_{w_i}\phi\frac{\partial (\frac{\phi_i}{\phi_r})}{\partial z_\alpha}v_\alpha = \sum_\alpha v_\alpha \zeta \text{ is parallel with } \phi(z_0)=\zeta.$$ 
As a result, $$K_v(\zeta)=\sum_{\alpha,\beta}\frac{\tilde{G} (R_{\alpha\bar{\beta}}\zeta,\zeta)}{\tilde{G}(\zeta,\zeta)}v_\alpha\bar{v}_\beta=\Theta|_{O_{P(E)}(-1)}(v,\bar{v}).$$

\end{proof}

One immediate consequence of Lemma \ref{lem inf} is that the Kobayashi curvature is seminegative if and only if $\log G(z,\phi(z))$ is plurisubharmonic for any local holomorphic section $\phi$ (which is also equivalent to $G$ being plurisubharmonic on the total space $E$).

\bibliographystyle{amsalpha}
\bibliography{Dominion}

\textsc{Institute of Mathematics, Academia Sinica, Taipei, Taiwan}

\texttt{\textbf{krwu@gate.sinica.edu.tw}}

\end{document}